\documentclass[11pt]{amsart}

\usepackage{amsmath,amsthm,amsfonts,amssymb}

\usepackage{color}

\newtheorem{thm}[equation]{Theorem}
\newtheorem{cor}[equation]{Corollary}
\newtheorem{lem}[equation]{Lemma}
\newtheorem{prop}[equation]{Proposition}

\theoremstyle{definition}

\numberwithin{equation}{section}

\newcommand{\ann}{\hbox{\rm Ann}}
\definecolor{mjo}{rgb}{0,0,.9}

\newcommand{\cc}{\mathbb{C}}
\newcommand{\z}{\mathbb{Z}}

\newcommand{\n}{\mathfrak{n}}
\newcommand{\h}{\mathfrak{h}}

\newcommand{\g}{\mathfrak{g}}
\newcommand{\p}{\mathcal{P}}

\newcommand{\W}{\mathcal{W}}

\definecolor{mjo}{rgb}{.4,0,.9}

\newcommand{\Hom}{\mbox{Hom}}

\begin{document}

\title{Whittaker Categories for the Virasoro Algebra}

\author{Matthew Ondrus}
\address{Mathematics Department \\ Weber State University \\ Ogden, Utah  84408-1702}
\email{mattondrus@weber.edu}
\author{Emilie Wiesner}
\address{Department of Mathematics \\
           Ithaca College\\ Ithaca, New York 14850  }
\email{ewiesner@ithaca.edu}


\maketitle

\begin{abstract}
This paper builds on work from \cite{OW2008}, where the authors described Whittaker modules for the Virasoro algebra.  Using the framework outlined in \cite{BatMaz}, the current paper investigates a category of Virasoro-algebra modules that includes Whittaker modules.  Results in this paper include a classification of the simple modules in the category and a description of certain induced modules that are a natural generalization of simple Whittaker modules.  
\end{abstract}

\section{Introduction}

Whittaker modules were first defined for $\mathfrak{sl}_2(\cc)$ by Arnal and Pinzcon \cite{AP}.  Motivated by the Whittaker equations in number theory, Kostant \cite{Ko78} later defined Whittaker modules for all finite-dimensional complex semisimple Lie algebras $\g$.  Among other results on Whittaker modules,  Kostant classified the Whittaker modules of $\g$, demonstrating a strong connection with the center of $U(\g)$.

Kostant's definition of Whittaker modules is closely tied to the triangular decomposition of a finite-dimensional complex semisimple Lie algebra: $\g= \n^- \oplus \h \oplus \n^+$ ; each Whittaker module depends upon a fixed nonsingular Lie algebra homomorphism $\psi : \n^+ \to \cc$.  Results for complex semisimple Lie algebras have been extended to other algebras with similar structures.  These include quantum groups, by Sevoystanov \cite{Se00} for $U_h(\g)$ and by Ondrus \cite{On05} for $U_q(\mathfrak{sl}_2)$.  Whittaker modules have also been studied for generalized Weyl algebras  by Benkart and Ondrus \cite{BO08} and in connection to non-twisted affine Lie algebras by Christodoupoulou \cite{Ch08}.  In \cite{OW2008}, the present authors studied Whittaker modules for the Virasoro algebra; analogous results in similar settings have been worked out in \cite{Wa10} and \cite{Li10}. 

Kostant's definition of Whittaker modules leads to a category that is not abelian, among other categorical limitations.  However, Whittaker modules do exist naturally inside of larger, better-behaved categories; we refer to these categories as Whittaker categories.   Work toward understanding these larger categories in the setting of complex semisimple Lie algebras was begun by McDowell \cite{M85}, Mili{\v{c}}i{\'c} and Soergel \cite{MS97}, and Backelin \cite{Ba97}.   Batra and Mazorchuk \cite{BatMaz} later generalized the ideas of both Whittaker modules and the underlying categories to a broad class of Lie algebras.  Their framework allows for a unified explanation of important results but is also limited by its generality.  In particular, a lack of specific knowledge about the center of $U(\g)$ appears to contribute to the difficulty in understanding some aspects of Whittaker categories in general.
 
In this paper, we make use of the framework created by Batra and Mazorchuk to further explore Whittaker categories in the setting of the Virasoro algebra.  The Virasoro algebra has been widely studied due in part to its interesting representation theory and its role in mathematical physics.  Knowledge of the center of the universal enveloping algebra of the Virasoro algebra also allows for a more detailed study of Whittaker categories in this setting.  

Using the triangular decomposition $\g=\n^- \oplus \h \oplus \n^+$ of the Virasoro algebra, we investigate the category $\W$ of all $\g$-modules on which $\n^+$ acts locally finitely.  Section \ref{sec:WhittakerModules} contains a review of relevant definitions and results from \cite{OW2008} regarding Whittaker modules for the Virasoro algebra.  In Section \ref{sec:categoryW}, we prove several general results about the category $\W$.   Theorem \ref{thm:quotU+LocFin} shows that the category $\W$ is a Serre subcategory of $\g$-Mod, and Proposition \ref{prop:simplePsiNot0} and Theorem \ref{lem:simplesInN(0)} demonstrate that the simple modules in $\W$ are exactly the simple lowest weight modules described in \cite{MP95} and the simple Whittaker modules of \cite{OW2008}.

In Section \ref{sec:categoryWf}, we study the subcategory $\W_f$ of $\W$ containing modules with finite composition length.  We give a decomposition of $\W_f$ according to the action of the center $Z$ of $U(\g)$ (Proposition \ref{prop:zaction}) as well as a partial description of homomorphisms between modules in $\W_f$.  In light of the decomposition given by the $Z$-action, it is natural to study modules in $\W_f$ on which $Z$ acts by scalars.  We prove several general facts related to the Whittaker vectors of these modules and also give a construction that appears to yield a significant subset of these modules.  In Theorem \ref{thm:indecompSupDiag}, we show that under certain conditions this construction produces modules of arbitrary composition length with a 1-dimensional space of Whittaker vectors.  This cannot happen in the setting of complex semisimple finite-dimensional Lie algebras (see Theorem 4.3 of \cite{Ko78}).

\section{Whittaker modules for the Virasoro algebra} \label{sec:WhittakerModules}

Let $\g$ denote the Virasoro Lie algebra.  Then $\g$ has a $\cc$-basis $\{z, d_k \mid k \in \z \}$
with Lie bracket 
\begin{align*}
[d_k, d_j] &= (k-j)d_{k+j} + \delta_{j,-k} \frac{k^3-k}{12} z; \\
[z, d_k] &= 0.
\end{align*}
Let $U = U(\g)$ denote the universal enveloping algebra of $\g$ and $Z = Z(\g)$ the center of $U(\g)$.  Note that $Z(\g) \cong \cc [z]$. (This follows, for example, from  \cite[Corollary 5.2]{OW2008}).

The Virasoro algebra has a triangular decomposition (in the sense of \cite{MP95}):
$$
\g=\n^-\oplus \h \oplus \n^+ 
$$
where $\n^{\pm}={\rm span}_\cc \{d_{\pm k} \mid k \in \z_{>0} \}$ and $\h={\rm span}_\cc \{ d_0, z\}$.  
For a given $\psi \in (\n^+/[\n^+,\n^+])^*$, define the 1-dimensional $\n^+$-module $\cc_\psi$ by $x.c=\psi(x)c$ for $x \in \n^+$ and $c \in \cc$.  As Lemma  \ref{lem:quasiNil_1dim} shows, every finite-dimensional simple $\n^+$-module has this form; the result is a consequence of Lie's Theorem and is proven for all quasi-nilpotent Lie algebras in \cite{BatMaz}.

\begin{lem}[\cite{BatMaz}]\label{lem:quasiNil_1dim}
Let $L$ be a simple, finite-dimensional $\n^+$-module.  Then $L \cong \cc_{\psi}$ some $\psi \in (\n^+/[\n^+,\n^+])^*$.
\end{lem}

A vector $w \in V$ is a {\it Whittaker vector} if there exists $\psi \in (\n^+/[\n^+,\n^+])^*$ such that $xw = \psi(x) w$ for all $x \in \n^+$.  A $\g$-module $V$ is a {\it Whittaker module} if there is a Whittaker vector $w \in V$ that generates $V$. In \cite{OW2008}, the authors give a characterization of Whittaker modules where $\psi \neq 0$. (We address related results for $\psi=0$ in this paper.) We repeat two results from \cite{OW2008} that will be central to the ideas described here.

For $\psi \in (\n^+/[\n^+, \n^+] )^*$ and $\xi \in \cc$, define $\cc_{\psi, \xi}$ to be the one-dimensional $\n^+ \oplus \cc z$-module on which $\n^+$ acts by $\psi$ and $z$ acts by $\xi$.  Then define the $\g$-module
$$
L(\psi, \xi) = {\rm Ind}_{\n^+ \oplus \cc z}^{\g} \cc_{\psi, \xi} = U(\g) \otimes_{U(\n^+ \oplus \cc z)} \cc_{\psi, \xi}.
$$

\begin{prop}[\cite{OW2008}] \label{cor:SimpleL_PsiXi} 
For any $0 \neq \psi \in (\n^+/[\n^+, \n^+] )^*$ and $\xi \in \cc$, the module $L(\psi, \xi)$ is simple.  Moreover, any Whittaker module where $z$ acts by a scalar is isomorphic to some $L(\psi,\xi)$, and $L(\psi,\xi) \cong L(\psi', \xi')$ if and only if $(\psi,\xi) = (\psi', \xi')$.
\end{prop}

For a $\g$-module $V$, we let $\ell (V)$ denote the composition length (possibly $\infty$) of $V$.

\begin{prop}[\cite{OW2008}] \label{thm:generalV}
Let $0 \neq \psi \in (\n^+/[\n^+, \n^+] )^*$, and let $V$ be a Whittaker module of type $\psi$ with cyclic Whittaker vector $w$.  Suppose that $\ann_{Z}(w) \neq 0$.  Let $p(z)$ be the unique monic generator of the ideal $\ann_{Z}(w)$ in $Z$, and write $p(z) = \prod_{i=1}^k (z- \xi_i1)^{a_i}$ for distinct $\xi_1, \ldots, \xi_k \in \cc$.   Then $V$ decomposes into a direct sum of Whittaker modules
$$
V= \oplus_{i=1}^k V_i
$$
where $\ell (V_i)= a_i$ and the composition factors of $V_i$ are all isomorphic to $L(\psi, \xi_i)$.
\end{prop}

\section{The category $\W$} \label{sec:categoryW}
Batra and Mazorchuk \cite{BatMaz} have generalized the idea of Whittaker modules to a larger category.  Adopting their definition, we define the {\it Whittaker Category} $\W$ to be the full subcategory of $\g$-Mod containing $\g$-modules on which the action of $\n^+$ is locally finite.  

Observe that $\W$ is an abelian category.  Here we also show that this category is closed under taking extensions; thus it forms a Serre subcategory of $\g$-Mod. Theorem \ref{thm:quotU+LocFin} is similar to \cite[Proposition 1]{BatMaz}.  However, the proof in \cite{BatMaz} relies on the assumption that $\n^+$ is finite-dimensional and thus $U(\n^+)$ is left Noetherian. 

\begin{thm}\label{thm:quotU+LocFin}
Suppose that $0 \to X \to Y \to Z \to 0$ is an exact sequence with $X, Y, Z \in {\rm Ob}( \mathfrak g\text{-${\rm Mod}$})$.   If $X, Z \in {\rm Ob}(\W)$, then $Y \in {\rm Ob}(\W)$.   
\end{thm}
We note that the proof of this result can be applied to any situation where $U(\n^+)$ is finitely generated.
\begin{proof}
To prove the claim, we show that $U(\n^+)y_0$ is finite-dimensional for any $y_0 \in Y$.  

Let $- : Y \to Y/X$ denote the natural homomorphism.  Since $Y/X \cong Z \in {\rm Ob}(\mathcal W)$, we know that $U(\n^+) \overline{y_0}$ is a finite-dimensional subspace of $Y/X$.  Let $y_1, \ldots, y_n \in Y$ so that $U(\n^+) \overline{y_0} = {\rm span}_\cc \{ \overline{y_1}, \ldots, \overline{y_n} \}$.

Since $\overline{y_0} \in U(\n^+) \overline{y_0} = {\rm span}_\cc \{ \overline{y_1}, \ldots, \overline{y_n} \}$, there exist $c_1, \ldots, c_n \in \cc$ such that 
$$\overline{y_0} = \sum_{1 \le k \le n} c_k \overline{y_k},$$
and thus
$$y_0 = x_0 + \sum_{1 \le k \le n} c_k y_k,$$
for some $x_0 \in X$.  Similarly, since $U(\n^+) \overline{y_0}$ is $\n^+$-invariant, for every $i>0$ and $j \in \{ 1, \ldots, n \}$, we have
$$d_i \overline{y_j} = \sum_{1 \le k \le n}  c_{i,j,k} \overline{y_k}$$
for some $c_{i,j,k} \in \cc$; consequently there is some $x_{i,j} \in X$ such that 
$$d_i y_j = x_{i,j} +  \sum_{1 \le k \le n}  c_{i,j,k} y_k.$$

Let 
\begin{align}\label{eqn:defSubspaceN}
N = U(\n^+) x_0 + \sum_{\substack{i=1,2 \\ 1 \le j \le n}} U(\n^+) x_{i,j},
\end{align}
and define $M = \cc y_1 + \cdots + \cc y_n + N.$
Since $X \in {\rm Ob}(\mathcal W)$, each $U(\n^+)x_{i,j}$ is finite-dimensional and $U(\n^+)x_0$ is finite-dimensional, and so $N$ is finite-dimensional.  Therefore $M$ is a finite-dimensional vector space containing $y_0 = x_0 + \sum_{1 \le k \le n} c_k y_k$.  

It now suffices to check that $U(\n^+)M \subseteq M$.  If $n \in N \subseteq M$ and $u^+ \in U(\n^+)$, then $u^+n \in N \subseteq M$.  To prove that $u^+ y_j \in M$ whenever $u^+ \in U(\n^+)$ and $j \in \{ 1, \ldots, n \}$, it is enough to show that $d_i y_j \in M$ whenever $j \in \{ 1, \ldots, n \}$ and $i>0$.  It is obvious that $d_i y_j \in M$ if $i = 1$ or $i = 2$ since the sum in (\ref{eqn:defSubspaceN}) is taken over $i=1,2$.  If $i>2$, then the result follows from the fact that $U(\n^+)$ is generated over $\cc$ by $d_1$ and $d_2$.
\end{proof}

As shown in Lemma \ref{lem:quasiNil_1dim}, the simple $\n^+$-modules are indexed by the set $(\n^+/[\n^+,\n^+])^*$.
Batra and Mazorchuk \cite{BatMaz} showed that, for a large class of Lie algebras, $\W$ decomposes into subcategories $\W(\psi)$,  labeled by $\psi \in (\n^+/[\n^+,\n^+])^*$. This decomposition provides a framework for the remainder of this paper, and so we present it here.   

\begin{thm}[\cite{BatMaz}] \label{thm:Wdecomp}
 Let $V \in \W$. Then,
\begin{itemize}
\item[(i)] for $\psi \in (\n^+/[\n^+, \n^+])^*$, the set 
$$
V^{\psi} = \{ v \in V \mid (x-\psi(x))^kv=0 \ \mbox{for $x \in \n^+$ and $k>>0$}\}
$$  
is a $\g$-submodule of $V$;
\item[(ii)] $V= \oplus_{\psi} V^{\psi}$;
\item[(iii)] if $\psi \neq \nu$, then $\Hom_{\g}(V^{\psi}, V^{\nu})=0$;
\item[(iv)] every $\n^+$-submodule of $V^\psi$ contains a nonzero $\psi$-Whittaker vector.
\end{itemize}
\end{thm}
\begin{proof}
In \cite{BatMaz}, Batra and Mazorchuk define the submodules in the decomposition using a more general construction.  However, for the Virasoro algebra the result simplifies to the present statement.  This simplification follows from Lemma \ref{lem:quasiNil_1dim} and the fact that the equivalence relation of \cite[p.~7]{BatMaz} is trivial for quasi-nilpotent algebras. Statement (iv) follows from Lemma \ref{lem:quasiNil_1dim} and Proposition 5 of \cite{BatMaz}.  


\end{proof}

In light of this decomposition, we  define $\W (\psi)$, $\psi \in \left( \n^+ / [\n^+, \n^+] \right)^*$,  as the subcategory of $\W$ consisting of all objects $V \in {\rm Ob} (\W)$ so that $V=V^{\psi}$. 

\subsection{Simple Objects in $\W$}\label{subsec:simpObjW}
 Theorem \ref{thm:Wdecomp}(ii) implies that to determine the simple modules in $\W$, it is enough to determine the simple objects in $\W(\psi)$ for each $\psi \in (\n^+/[\n^+, \n^+])^*$.  We show that the simple objects in $\W$ are either Whittaker modules (for $\psi \neq 0$) or lowest weight modules (for $\psi =0$). 
We first consider the case $\psi \neq 0$. 

\begin{prop}\label{prop:simplePsiNot0}
Let $0 \neq \psi\in (\n^+/[\n^+, \n^+])^*$, and suppose $V \in \W(\psi)$ be a simple object.  Then $V \cong L(\psi, \xi)$ for some $\xi \in \cc$.
\end{prop}
\begin{proof}
From Theorem \ref{thm:Wdecomp}(iv), there is an $\psi$-eigenvector $w \in V$.  Since $V$ is simple, we must have that $V$ is in fact a Whittaker module.  Therefore, the result follows from \cite{OW2008}.
\end{proof}

The case $\psi=0$ is handled in Proposition \ref{lem:simplesInN(0)}.   We begin by presenting the construction of several relevant modules. An element $\alpha \in \h^*$ can be identified with a pair $(\xi, h) \in \cc \times \cc$, where $\alpha(d_0)=h$ and $\alpha(z)=\xi$.  Then the Verma module of lowest weight $\alpha=(\xi, h)$ is given by
$$
M(0,\xi,h)= U(\g) \otimes_{U(\h\oplus\n^+)} \cc_{(\xi, h)}
$$ where $\h$ acts on the one-dimensional $\h \oplus \n^+$-module $\cc_{(\xi, h)}$ by $(\xi, h)$ and $\n^+$ acts by $0$.  Then $M(0,\xi,h)$ has a unique simple quotient $L(0,\xi,h)$.  Moreover, all simple lowest weight modules are isomorphic to some $L(0,\xi,h)$ (cf. \cite{MP95}). We also define a ``universal" module 
$$
M(0,\xi)= U(\g) \otimes_{U(\cc z\oplus\n^+)} \cc_{\xi}
$$ where $z$ acts on $\cc_{\xi}$ by $\xi$ and $\n^+$ acts by $0$. 

By Theorem \ref{thm:Wdecomp}(iv), a simple module $V \in \W(0)$ is necessarily generated by a vector $w\in V$ such that $\n^+$ acts by $0$ (that is, a Whittaker vector for $\psi=0$). However, we must also show that $w$ can be taken to be a weight vector.  

In order to write the following results more concisely, we adopt some notation.  Define a {\it partition} $\mu$ to be a non-decreasing sequence of positive integers $\mu=( 0 < \mu_1 \leq \mu_2 \leq \cdots \leq \mu_r)$.  Let $\p$ represent the set of partitions. 
Also define
\begin{eqnarray*}
| \mu | &=& \mu_1 +  \mu_2 + \cdots + \mu_r \quad \mbox{(the size of $\mu$)}\\
\# \mu &=& r \quad \mbox{(the $\#$ of parts of $\mu$)}\\
d_{ \mu} &=& d_{\mu_1} d_{\mu_2}\cdots d_{\mu_r} = d_0^{\mu(0)} d_1^{\mu(1)} \cdots \\
d_{-\mu} &=& d_{-\mu_r} \cdots d_{-\mu_2}  d_{-\mu_1}= \cdots d_{-1}^{\mu(1)} d_0^{\mu(0)}\\
d_{-\mu} &=& d_{\mu} =1 \quad \mbox{for  $\mu = \emptyset$}.
\end{eqnarray*}

The first lemma is evident by the construction of $M(0,\xi)$ along with the PBW Theorem. 
\begin{lem}
The module $M(0,\xi)$ has a basis
$$\{ d_{-\lambda} w_i \mid \lambda \in \mathcal P, \, i \ge 0 \},$$
where $w_i= d_0^i \otimes 1$.  Moreover,  
$$d_0w_i = w_{i+1}, \qquad zw_i = \xi w_i, \quad d_n w_i = 0 \quad (n>0)$$
for $i \ge 0$ and $M(0,\xi)$ is generated by $w_0$ as a $U(\g)$-module.  
\end{lem}

\begin{lem}
Let $\lambda \in \p$ and $j, n \in \z_{>0}$.  
Then $d_n d_{-\lambda} w_j$ is a linear combination of vectors in the following set:
$$
\{ d_{-\gamma} w_j \mid \# \gamma \leq \# \lambda \} \cup \{ d_{-\gamma} w_{j+1} \mid \# \gamma < \# \lambda -1 \} \cup \{ d_{- \hat{\lambda}} w_{j+1} \mid \hat{\lambda}= ( \lambda_1, \ldots, \hat{\lambda}_i, \ldots) \} .
$$
Moreover, the coefficient of $d_{- \hat{\lambda}} w_{j+1}$ is nonzero if and only if $n= \lambda_i$.
\end{lem}
\begin{proof}
We use induction on $\# \lambda$, with the case $\#\lambda=1$ being obvious.

Now suppose $\# \lambda >1$.  We have
\begin{eqnarray*}
d_n d_{-\lambda} w_j &=& d_{-\lambda} d_n w_j + \sum_i d_{-\lambda_1} \cdots [d_n, d_{-\lambda_i}] \cdots d_{-\lambda_s} w_j.
\end{eqnarray*}
Note that $d_{-\lambda} d_n w_j=0$.  For the second term, we consider three cases.
\begin{itemize}
\item[(i)] If $n>\lambda_i$, $[d_n, d_{-\lambda_i}]  =d_{n-\lambda_i}$ with $n-\lambda_i >0$.  By the inductive hypothesis, $[d_n, d_{-\lambda_i}] \cdots d_{-\lambda_s} w_j$ is contained in the span of
$$
\{ d_{-\gamma} w_j \mid \# \gamma \leq s-i \} \cup \{ d_{\gamma} w_{j+1} \mid \# \gamma <  s-i-1 \} \cup \{ d_{- \hat{\lambda}} w_{j+1} \mid \hat{\lambda}= ( \lambda_{i+1} \ldots, \hat{\lambda}_k, \ldots) \} .
$$
which implies that $d_{-\lambda_1} \cdots [d_n, d_{-\lambda_i}] \cdots d_{-\lambda_s} w_j$ is contained in the span of 
$$
\{ d_{-\gamma} w_j \mid \# \gamma \leq \# \lambda \} \cup \{ d_{-\gamma} w_{j+1} \mid \# \gamma < \# \lambda -1 \}.
$$

\item[(ii)] If $n<\lambda_i$, then $[d_n, d_{-\lambda_i}]  =d_{n-\lambda_i}$ with $n-\lambda_i <0$.  This implies that 
$d_{-\lambda_1} \cdots [d_n, d_{-\lambda_i}] \cdots d_{-\lambda_s} w_j$ is contained in the span of 
$$
\{ d_{-\gamma} w_j \mid \# \gamma \leq \# \lambda \}.
$$

\item[(iii)] If $n = \lambda_i$, then $[d_n, d_{-\lambda_i}]  =2n d_0+ cz$ for some $c \in \cc$.  Then,
$$
d_{-\lambda_1} \cdots [d_n, d_{-\lambda_i}] \cdots d_{-\lambda_s} w_j = 2 n d_{-\hat{\lambda}} w_{j+1} + (2n(\lambda_{i+1} + \cdots + \lambda_s)+ c \xi) d_{-\hat{\lambda}} w_{j}.
$$
Note that if $\lambda_i$ appears with multiplicity $k$, then the coefficient of $d_{-\hat{\lambda}}w_{j+1}$ is $2kn>0$.
\end{itemize}
Combining these cases proves the result.
\end{proof}

Let $\lambda, \mu \in \p$, and write $\lambda = (\lambda_1, \ldots, \lambda_n)$ and $\mu=(\mu_1, \ldots, \mu_k)$.  If $\mu_j=\lambda_{i_j}$ for some sequence $i_1, \ldots, i_k$ of distinct integers, then define 
$$\lambda-\mu=(\lambda_1, \ldots, \hat{\lambda}_{i_1}, \ldots,  \hat{\lambda}_{i_k}, \ldots, \lambda_n).$$  If it is not the case that $\mu_j=\lambda_{i_j}$ for some sequence $i_1, \ldots, i_k$ of distinct integers, then we regard $\lambda - \mu$ as undefined.  

\begin{lem} \label{Lem:wj}
Let $j \in \z_{\geq 0}$, $\mu, \lambda \in \p$.  Then $d_{\mu} d_{-\lambda} w_j$ is a linear combination of vectors in the following set 
$$
\{ d_{-\gamma} w_j \mid \# \gamma \leq \# \lambda \} \cup \{ d_{-\gamma} w_{j+k} \mid 0<k \leq \# \mu, \# \gamma < \# \lambda -k  \} \cup \{ d_{- (\lambda-\mu)} w_{j+\# \mu} \},
$$
where the coefficient of $d_{- (\lambda-\mu)} w_{j+\# \mu}$ is nonzero if and only if $\lambda-\mu$ is defined. 
\end{lem}
\begin{proof}
We induct on $\# \mu$.  The previous lemma proves the case $\# \mu=1$.  For the inductive step, write $\mu = (\mu_1, \ldots, \mu_r)$, and let $\mu'=(\mu_1, \ldots, \mu_{r-1} )$. Using the previous lemma, we have
\begin{eqnarray*}
d_{\mu} d_{-\lambda} &=& d_{\mu'} \left( d_{\mu_r} d_{-\lambda} w_j \right) \\
&=& d_{\mu'}  \sum_{\gamma \mid \# \gamma \leq \# \lambda } a_{\gamma} d_{-\gamma} w_j \\
&& + d_{\mu'} \sum_{\gamma \mid \# \gamma < \# \lambda -1} b_{\gamma} d_{-\gamma} w_{j+1} \\
&& + c d_{\mu'} d_{- \hat{\lambda}} w_{j+1},
\end{eqnarray*}
where $c \neq 0$ if and only if $\mu_r=\lambda_i$. Since $\# \mu'= \# \mu -1$, we can now apply the inductive argument to get the result.
\end{proof}

Continue to use the notation $w_0, w_1, \ldots \in M(0,\xi)$ from above, and define the following subspace of $M(0,\xi)$:
\begin{align}\label{eqn:defOfW}
W_\xi=\mbox{span}_{\cc}\{ w_i \mid i \in \z_{\ge 0} \}.
\end{align}

\begin{lem}\label{Lem:submodMxi}
Let $V \subseteq M(0,\xi)$ be a submodule. Then $V \cap W_\xi \neq 0$.
\end{lem}
\begin{proof}
If $v \in V$, then 
$$v=\sum_{\genfrac{}{}{0pt}{}{\lambda \in \p}{i \ge 0}} a_{\lambda, i} d_{-\lambda} w_i$$
for some $a_{\lambda, i} \in \cc$.  Let $N$ be maximal such that $a_{\lambda, i} \neq 0$ for some $i \ge 0$ and $|\lambda|=N$.  Now choose $\gamma, j$ such that $a_{\gamma,j} \neq 0$, $|\gamma|=N$, and $\#\gamma+j$ is maximal.  (Note that there may not be a unique choice of $\gamma$ and $j$.) 

Now,
\begin{eqnarray*}
d_{\gamma} v &=& a_{\gamma, j} d_{\gamma} d_{-\gamma} w_j +   \sum_{\ i < j}  a_{\gamma, i} d_{\gamma} d_{-\gamma} w_i \\
& &
+ \sum_{|\lambda| =N, \lambda \neq \gamma}  a_{\lambda, i} d_{\gamma} d_{-\lambda} w_i 
+  \sum_{|\lambda| <N}  a_{\lambda, i}d_{\gamma}  d_{-\lambda} w_i. 
\end{eqnarray*}

The last sum is zero since $d_\gamma d_{-\lambda} w_i = 0$ whenever $|\gamma|>|\lambda|$.  For the other terms, note that if $|\gamma|=|\lambda|$, then $d_{\gamma} d_{-\lambda} \in U(\g)_0$.  (Here, $U(\g)_0$ represents the $0$-weight space of $U(\g)$ under the adjoint action of $\h$.)  Since $d_n w_j=0$ for $n>0$, the PBW Theorem implies that $d_{\gamma} d_{-\lambda} w_j = x w_j$ for some $x \in U(\h)$.  In this case, $d_{\gamma} d_{-\lambda} w_j \in W_\xi$.  Combining this with Lemma \ref{Lem:wj}, we have 
\begin{align}\label{eqn:submodMxiFinite}
 \sum_{\ i < j}  a_{\gamma, i} d_{\gamma} d_{-\gamma} w_i + \sum_{|\lambda| =N, \lambda \neq \gamma}  a_{\lambda, i} d_{\gamma} d_{-\lambda} w_i  \in \, {\rm span}_\cc \{ w_k \mid k< j+ \#\gamma \}.
\end{align}
Similarly, Lemma \ref{Lem:wj} implies that the coefficient of $w_{j+\#\gamma}$ in $d_{\gamma} d_{-\gamma} w_j $ is nonzero, and this term is not cancelled by any of the terms in (\ref{eqn:submodMxiFinite}).
\end{proof}

\begin{prop}\label{lem:simplesInN(0)}
If $V \in {\rm Ob}(\W(0))$ is simple, then $V\cong L(0, \xi, h)$ for some $\xi, h \in \cc$.
\end{prop}

\begin{proof}
It is enough to find a lowest weight vector: $0 \neq v_0^+ \in V$ so that $\n^+$ annihilates $v^+_0$ and $d_0, z$ act on $v^+_0$ by scalars.  By Theorem \ref{thm:Wdecomp}(iv), $V$ is generated by a vector $v_0$ that is annihilated by $\n^+$.  Moreover, since $V$ is simple, $z$ acts on $V$ by a scalar.  (Cf. \cite[Ex.~2.12.28]{rowen:rt88} for an appropriate generalization of Schur's Lemma.)  If some polynomial $(d_0 - a_1)(d_0 - a_2) \cdots (d_0 - a_k)$ in $d_0$ (for $a_1, \ldots, a_k \in \cc$) of minimal degree anihilates $v_0$, then clearly $v_0^+ = (d_0 - a_2) \cdots (d_0 - a_k)v_0$ is a nonzero lowest weight vector.  Therefore it is sufficient to show that there is some nonzero polynomial in $d_0$ that annihilates $v_0$.

Define $w_0 = 1 \otimes 1 \in M(0,\xi)$, and $v^+ = 1 \otimes 1\in M(0,\xi,0)$ (a lowest weight vector).  It is straightforward to verify that $d_{-1} v^+$ is also a lowest weight vector and thus $M(0,\xi, 0)$ is not simple.  Since there is a surjective module homomorphism $\varepsilon_0 : M(0,\xi) \to M(0,\xi,0)$ determined by $w_0 \mapsto v^+$, it must be that $M(0,\xi)$ is not simple.  Similarly, there is a surjective module homomorphism $\varepsilon: M(0,\xi) \rightarrow V$ defined by $w_0 \mapsto v_0$.  The kernel of this map is necessarily nonzero since $V$ is (by assumption) simple and $M(0,\xi)$ is not.  Let $S = \ker \varepsilon \neq 0$ so that $V \cong M(0,\xi) / S$.  

Let $W_\xi$ be as in (\ref{eqn:defOfW}), and observe that $S \cap W_\xi \neq 0$ by Lemma \ref{Lem:submodMxi}.  This means that there exists a nonzero vector in $S$ of the form $q(d_0)w_0$ for some polynomial $q$.  But then we see that 
$$0 = \varepsilon (q(d_0)w_0) = q(d_0) \varepsilon (w_0) = q(d_0)v_0,$$
as desired. 
\end{proof}

\section{The category $\W_f$} \label{sec:categoryWf}
Let $\W_f$ be the full subcategory of $\W$ containing $\g$-modules with finite composition length, and define $\W_f(\psi)=\W_f \cap \W(\psi)$.
Restricting to the subcategory $\W_f$ in this section allows us to use the center $Z(\g)$ to better describe modules in the category. In some cases we further restrict to $\W_f(\psi)$ where $\psi \neq 0$.  However, it follows from Proposition \ref{lem:simplesInN(0)} that $\W_f(0)$ is a subcategory of the category $\mathcal{O}$, about which much is already known for the Virasoro algebra.

Note that any module in $\W_f$ is finitely generated and has a locally finite action of $Z(\g)$.   (A generalization of Schur's Lemma (see \cite[Ex.~2.12.28]{rowen:rt88}) implies that $Z(\g)$ acts by a scalar on any simple subquotient of $V \in \W_f)$.  Therefore, the following proposition shows that $\W_f$ decomposes by central characters. This decomposition is the basis for much of the work in this section.

\begin{prop} \label{prop:zaction}
Let $V \in {\rm Ob} (\W)$, and assume that the action of $Z(\g)\cong \cc[z]$ is locally finite.  Then, for $\xi \in \cc \cong {\rm Hom}_{\rm alg}(Z(\g),\cc)$, 
$$V_\xi = \{ v \in V \mid (z-\xi)^k v=0 \ \mbox{for $k>>0$} \}$$
is a $\g$-module; and  $$V = \bigoplus_{\xi \in \cc} V_\xi.$$
Moreover, if $V$ is finitely generated, then $V_{\xi}$ is finitely generated and has a finite filtration $0 = V_{\xi, 0} \subseteq V_{\xi, 1} \subseteq V_{\xi, 2} \subseteq \cdots \subseteq V_{\xi, k} = V$ such that $z$ acts by $\xi$ on $V_{\xi, i+1} / V_{\xi, i}$ for each $0 \leq i <k$, where 
$$V_{\xi, i} = \{ v \in V_\xi \mid (z - \xi)^iv = 0\}.$$
\end{prop}
\begin{proof}
It is clear that $V_{\xi}$ is a submodule of $V$ and that the sum $\sum_{\xi \in \cc} V_\xi$ is direct. It remains to show that every element $v \in V$ can be written as a sum $v = \sum_\xi v_\xi$ with $v_\xi \in V_\xi$.  This is a standard argument that follows from the fact that $\cc[z]v$ is finite-dimensional, and thus $\ann_{\cc[z]}(v) \neq 0$.


Now suppose that $V$ is finitely generated.  Since the submodule $V_\xi$ is a direct summand of $V$ and thus a homomorphic image of $V$, it follows that $V_\xi$ is also finitely generated.  For $i \in \z_{>0}$, let 
$$V_{\xi, i} = \{ v \in V_{\xi} \mid (z - \xi)^iv = 0\}.$$
Then the chain $0 \subseteq V_{\xi, 1} \subseteq V_{\xi, 2} \subseteq \cdots$ is a filtration of $V_{\xi}$.  Because $V$ is finitely generated, there is some $k$ such that $V_{\xi, k} = V_\xi$, and thus the filtration is finite.
\end{proof}
Note that even if $V \in {\rm Ob}( \mathcal W(\psi))$ for some $0 \neq \psi (\n^+/[\n^+, \n^+])^*$, the quotients $V_{\xi, i+1} / V_{\xi, i}$ in Proposition \ref{prop:zaction} may not be simple.

In light of Proposition \ref{prop:zaction}, we  define $\W_f (\psi, \xi)$, $\xi \in \cc$,  as the full subcategory of $\W_f(\psi)$ consisting of all objects $V \in {\rm Ob} (\W_f(\psi))$ so that $V=V_{\xi}$. We further study the structure of objects in $\W_f(\psi, \xi)$ in Section \ref{sec:modulesInW_f(psi)}.


\subsection{Module homomorphisms in Whittaker Categories} \label{sec:homomorphisms}
We wish to study ${\rm Hom}_\g (V,W)$, where $V,W \in {\rm Ob}(\mathcal \W_f)$.   Using Theorem \ref{thm:Wdecomp}, we have that
\begin{eqnarray*}
{\rm Hom}_\g (V,W) &=& {\rm Hom} (\bigoplus_\psi V^\psi,  \bigoplus_\eta W^\eta) \\
&=& \bigoplus_{\psi, \eta} {\rm Hom}_\g (V^\psi, W^\eta) \\
&=&\bigoplus_{\psi} {\rm Hom}_\g (V^\psi, W^\psi)
\end{eqnarray*}
Therefore, it is enough to fix $\psi \in (\n^+/[\n^+,\n^+])^*$ and consider ${\rm Hom}_\g (V, W)$ for $V, W \in {\rm Ob}(\mathcal \W_f(\psi))$. 
Proposition \ref{prop:Whittakerhomomorphism} characterizes the $\g$-module homomorphisms between Whittaker modules in $\W_f(\psi)$.  

\begin{lem} \label{lem:Whittakerhomomorphism}
Let $0 \neq \psi \in (\n^+/[\n^+,\n^+])^*$.  Suppose $V, W \in {\rm Ob} (W_f(\psi))$ are Whittaker modules with cyclic Whittaker vectors $v \in V$, $w \in W$.  If $\varphi : V \to W$ is a nonzero $\g$-module homomorphism, then there exists $r(z) \in Z(\g)$ such that $\varphi (v) = r(z) w$ and $r(z) \ann_U(v) \subseteq \ann_U(w)$.  Furthermore, for any $r(z) \in Z(\g)$ with $r(z) \ann_U(v) \subseteq \ann_U(w)$, there is a unique map $V \to W$ with $v \mapsto r(z)w$.  
\end{lem}
Although they are not in the category $\mathcal W_f$, this lemma holds for the universal Whittaker modules defined in \cite[Sec. ~2]{OW2008}.
\begin{proof}
\cite[Corollary 5.2]{OW2008} states that the set of Whittaker vectors in $W$ is given by ${\rm Wh}_\psi (W)=S(z)w$. The lemma follows from this and the fact that the homomorphic image of a Whittaker vector must be a Whittaker (possibly equal to 0).
\end{proof}

\begin{prop}  \label{prop:Whittakerhomomorphism}
Let $V, W \in {\rm Ob}(W_f(\psi))$ be Whittaker modules with cyclic Whittaker vectors $v \in V$, $w \in W$.  Let $\ann_{Z(\g)}(v) = \left< p(z) \right>$ and $\ann_{Z(\g)}(w) = \left< q(z) \right>$. Then

$$
\Hom_{\g} (V, W) \cong \cc [z] / \left< gcd (p(z), q(z))\right>
$$
where $\overline{s}(z)  \in \cc[z]/ \left< gcd (p(z), q(z))\right>$ defines a homomorphism $\phi_s: V \rightarrow W$ given by 
$$v \mapsto s(z) \cdot \frac{q(z)}{gcd(p(z), q(z))} w.$$
\end{prop}

\begin{proof}
Using the notation of Lemma \ref{lem:Whittakerhomomorphism}, we first show that $r(z) \ann_U (v) \subseteq \ann_U(w)$ if and only if $ r(z) p(z) \in \left< q(z) \right>$. Since $\ann_U(v) = Up(z) + I$ and $\ann_U(w) = Uq(z) +I$ for some left ideal $I$ (see \cite[Corollary 5.1]{OW2008}), it follows that $\left< r(z) p(z) \right> \subseteq \left< q(z) \right>$ implies $r(z) \ann_U (v) \subseteq \ann_U(w)$.  Conversely, if $r(z) \ann_U(v) \subseteq \ann_U(w)$, then clearly $r(z) p(z) \in \ann_U(w)$.  Since $r(z)p(z) \in Z(\g)$, we see that $r(z) p(z) \in \ann_U(w) \cap Z(\g) = \ann_{Z(\g)}(w) = \left< q(z) \right>$.  

Combining this with Lemma \ref{lem:Whittakerhomomorphism}, we have that all module homomorphisms $V \rightarrow W$ have the form $\phi_s$, where $\phi_s(v) = s(z) \cdot \frac{q(z)}{gcd(p(z), q(z))} w$ for some $s(z) \in \cc[z]$.  Since $q(z)$ is the unique monic polynomial of minimal degree in $\cc[z] \cap \ann_U(w)$, it is clear that $s(z)$ induces the 0 map if and only if $gcd(p(z), q(z))$ divides $s(z)$.  This implies that the surjective map $\cc[z] \to \Hom_{\g} (V, W)$ given by $s(z) \mapsto \phi_s$ has kernel $\left< gcd (p(z), q(z))\right>$.
\end{proof}

\begin{cor}
Let $V, W \in Ob (W_f(\psi))$ be Whittaker modules with cyclic Whittaker vectors $v \in V$, $w \in W$.   Write $\ann_{Z(\g)}(w)=\left< q(z) \right>$, and suppose $\varphi : V \to W$ is a $\g$-module homomorphism with $\varphi (v) = r(z) w$.  Then $\varphi$ is surjective if and only if $\gcd (r(z),q(z)) = 1$.
\end{cor}
\begin{proof}
Note that ${\rm im} \, \varphi$ is the submodule generated by $r(z)w$, and recall that $Z(\g) \cong \cc [z]$.  If $\gcd (r(z),q(z)) = 1$, then there exist $a(z), b(z) \in Z(\g)$ such that $a(z)r(z) + b(z)q(z) = 1$.  Since $q(z)w = 0$, it follows that 
$$w = (a(z)r(z) + b(z)q(z))w = a(z)r(z)w,$$
and therefore $w \in Z(\g)r(z)w \subseteq U(\g) r(z)w$.

Conversely, if $\varphi$ is surjective, then $w \in {\rm im} \, \varphi$.  But this means that $w$ is a Whittaker vector in the Whittaker module generated by the Whittaker vector $r(z)w$.  By Corollary 5.2 of \cite{OW2008}, we may write $w = a(z)r(z)w$ for some $a(z) \in Z(\g)$.  This implies that $a(z)r(z) - 1\in \left< q(z) \right>$, and therefore $\gcd (r(z),q(z)) = 1$.
\end{proof}

\begin{cor}
Let $\psi \in (\n^+/[\n^+, \n^+])^*)$,  $\xi, \xi' \in \cc$, $V \in \W_f(\psi, \xi)$, $W \in \W_f(\psi, \xi')$.  Then ${\rm Hom}_{\g}(V,W)=0$ unless $\xi=\xi'$.   
\end{cor}

\begin{proof}
This follows from Proposition \ref{prop:Whittakerhomomorphism} and the fact that every module in $\W$ contains a submodule that is a Whittaker module.
\end{proof}

\subsection{$Z$-semisimple Modules in $\W_f(\psi, \xi)$}\label{sec:modulesInW_f(psi)} \label{sec:ZSemisimpleModules}
By Proposition \ref{prop:zaction}, the study of $\W(\psi)$ reduces to the study of $\W(\psi, \xi)$ as long as the action of $Z$ on a given module is locally finite.  In Corollary \ref{cor:Z-locallyFinite}, we show that the condition that $Z$ acts locally finitely on $V$ is nearly equivalent to the condition that $V$ has finite composition length, and thus for the remainder of the paper we investigate $\W_f(\psi, \xi)$ for $\psi \neq 0$.

By assumption, any $V \in {\rm Ob} (\mathcal W_f(\psi, \xi))$ has a finite composition series, and we have seen that the corresponding simple quotients must be Whittaker modules.  Thus, the description of Whittaker modules in \cite{OW2008} provides some understanding of $\W_f(\psi, \xi)$.  Alternatively, we have that $z-\xi$ acts locally nilpotently on $V$, so $V$ also has a finite filtration by modules in $\W_f(\psi, \xi)$ where $z$ acts by $\xi$.   In this section we study such modules. 

For the following lemma, note that since $V \in {\rm Ob} (\mathcal W_f(\psi, \xi))$ is finitely generated and $\n^+$ acts locally finitely, $V$ is necessarily generated by a finite-dimensional $\n^+$-module.

\begin{lem}\label{lem:compSeriesN}
Fix $0 \neq \psi \in (\n^+/[\n^+, \n^+])^*$, $\xi \in \cc$, and suppose $V \in {\rm Ob} (\W(\psi, \xi))$ so that $z$ acts by $\xi$.  Let $N \subseteq V$ be a finite-dimensional $\n^+$-submodule such that $V = U(\g)N$.  If $n = \dim_\cc N$, then there exist vectors $v_1, \ldots, v_n \in N$ with the following properties:
\begin{enumerate}
\item $v_1 \in {\rm Wh}_\psi (V)$
\item For $1 \le k \le n$, the space ${\rm span}_\cc \{ v_1, \ldots, v_k \}$ is an $\n^+$-submodule of $N$.
\item For $1 \le k \le n$, $(d_i - \psi_i) v_k \in {\rm span}_\cc \{ v_1, \ldots, v_{k-1} \}$ \, ($i = 1, 2$).
\item For $1 \le k \le n$, the space $V_k := \sum_{i=1}^k U(\g) v_i$ is a $\g$-submodule of $V$, and the quotient $V_k / V_{k-1}$ is either trivial or is a simple Whittaker module with cyclic Whittaker vector $v_k + V_{k-1}$.  
\end{enumerate}
In particular, the composition length of $V$ is at most $\dim_\cc N$.
\end{lem}
\begin{proof} 
We first define the vectors $v_1, \ldots, v_n$.  By Theorem \ref{thm:Wdecomp}, there exists a nonzero $\psi$-eigenvector $v_1 \in N$, and thus $U(\n^+)v_1 = \cc v_1$.  In general, if $v_1, \ldots, v_{k-1}$ are defined, then we may regard ${\rm span}_\cc \{ v_1, \ldots, v_{k-1} \} \subseteq N$ as $\n^+$-submodules of $V$.  The $\n^+$-module $N / {\rm span}_\cc \{v_1, \ldots , v_{k-1} \}$ is finite-dimensional, and contains a simple submodule.  By Lemma \ref{lem:quasiNil_1dim}, the simple submodule is 1-dimensional with $\n^+$ action given by $\psi$.  In other words, there exists some $v_k \in N \setminus {\rm span}_\cc \{ v_1, \ldots, v_{k-1} \}$ such that $v_k + {\rm span}_\cc \{ v_1, \ldots, v_{k-1} \}$ is a nonzero $\psi$-eigenvector in $N/ {\rm span}_\cc \{ v_1, \ldots, v_{k-1} \}$.  Thus $v_1, \ldots, v_n$ exist by induction.  With $v_1, \ldots, v_n$ now defined, the remaining assertions are straightforward to verify.
\end{proof}

\begin{cor}\label{cor:Z-locallyFinite}
Let $0 \neq \psi \in (\n^+/[\n^+, \n^+])^*$ and let $V \in {\rm Ob} (\W(\psi))$.  Then $V \in {\rm Ob} (\W_f(\psi))$ if and only if $V$ is finitely generated and $Z(\g)$ acts locally finitely on $V$.
\end{cor}
\begin{proof}
If $V \in \W_f$, then $V$ is finitely generated and $Z(\g)$ acts locally finitely.
Therefore suppose $V$ is finitely generated and $Z(\g)$ acts locally finitely.  By Proposition \ref{prop:zaction}, $V = \bigoplus_{\xi \in \cc} V_\xi$, so it is no loss to assume that $V = V_\xi$ for some $\xi \in \cc$.   But Proposition \ref{prop:zaction} also implies that $V$ has a finite filtration $0=V_0 \subset \cdots \subset V_k =V$ where $z$ acts by a scalar on $V_{i+1}/V_i$.  Using Lemma \ref{lem:compSeriesN}, this filtration can be refined to a (still finite) filtration by Whittaker modules with $z$ acting by a scalar.    It follows from \cite[Proposition 4.8]{OW2008} that the factors in this refined filtration are simple, and thus $V$ has a finite composition series.
\end{proof}

We observe that if $\psi=0$, not all finitely generated modules on which $Z(\g)$ acts locally finitely have finite composition length.  In particular, there are Verma modules which do not have finite composition length.  (Cf. \cite{FF90}.) 

The next results use $N$ to describe the structure of $V$.

\begin{prop}\label{prop:decompInN(psi)}
Let $V \in {\rm Ob} (\W_f(\psi, \xi))$ so that $z$ acts by $\xi$ ($\psi \neq 0$).  Let $N \subseteq V$ be a finite-dimensional $\n^+$-submodule such that $U(\g)N = V$.  Suppose $N = N_1 \oplus \cdots \oplus N_k$ so that $N_i$ is an $\n^+$-submodule and ${\rm Wh}_\psi (U(\g)N_i) \subseteq N_i$ for all $i$.  Then $V = \bigoplus_{i = 1}^k U(\g)N_i$. 
\end{prop}
\begin{proof}
The sum $\sum_{i = 1}^k U(\g)N_i$ is a $U(\g)$-submodule that contains $N$, so it is clear that $\sum_{i = 1}^k U(\g)N_i = V$.  It remains to show that the sum $\sum_{i = 1}^k U(\g)N_i$ is direct.  Suppose that 
\begin{align}\label{eqn:sumEqualsZero}
v_1 + \cdots + v_n = 0,
\end{align}
with $v_i \in U(\g)N_i$ for all $i$ and $v_i \neq 0$ for some $i$.  It is no loss to suppose that $v_1 \neq 0$.  By Theorem \ref{thm:Wdecomp}, there exists some $u_1 \in U(\n^+)$ such that $u_1 v_1$ is a nonzero $\psi$-eigenvector.  On the other hand, $u_1 v_1 \in U(\g)N_1$, and by assumption, we have ${\rm Wh}_\psi (U(\g)N_1) \subseteq N_1$.  

If we multiply (\ref{eqn:sumEqualsZero}) by $u_1$, we obtain 
\begin{align}\label{eqn:HitWith_u1}
n_1 + v_2' + v_3' + \cdots + v_n' = 0,
\end{align}
where $0 \neq n_1 \in N_1 \cap {\rm Wh}_\psi (V)$ and $v_i' \in U(\g)N_i$ for $i = 2, \ldots, n$.  Now if $v_2' = v_3' = \cdots = v_n' = 0$, then we have the contradiction that $n_1 = 0$.  So suppose, without loss of generality, that $v_2' \neq 0$.  Then by a similar argument (to above), there exists $u_2 \in U(\n^+)$ such that $0 \neq n_2 := u_2 v_2' \in {\rm Wh}_\psi (U(\g)N_2) \subseteq N_2$.  Note that because $n_1 \in {\rm Wh}_\psi (V)$ and $u_2 \in U(\n^+)$, it follows that $u_2 n_1 = c_1 n_1$ for some $c_1 \in \cc$.  Multiply (\ref{eqn:HitWith_u1}) by $u_2$ to obtain 
$$c_1n_1 + n_2 + v_3'' + \cdots + v_n'' = 0,$$
where $v_i'' \in U(\g)N_i$ for $i = 3, \ldots, n$.  If we apply this argument repeatedly, we are eventually able to write $0$ as a sum of vectors in the various $N_i$, where at least one component is nonzero.  This contradicts the fact that the sum $N = \bigoplus N_i$ is direct.  Thus our assumption that one or more of the $v_i$ in (\ref{eqn:sumEqualsZero}) is nonzero must be wrong, so that the sum $\sum_i U(\g)N_i$ is direct.
\end{proof}

\begin{cor}\label{cor:whittVectsDecomp}
Let $V \in {\rm Ob} (\W_f(\psi))$ ($\psi \neq 0$), and suppose that $z$ acts by a scalar.  If $V$ is generated by linearly independent Whittaker vectors $w_1, \ldots, w_k$, then $V = \bigoplus_{i = 1}^k U(\g)w_i$.  
\end{cor}

\begin{proof}
Note that the space $N = \cc w_1 + \cdots + \cc w_k$ is an $\n^+$-submodule of $V$ that decomposes into $\n^+$-submodules as $N = N_1 \oplus \cdots \oplus N_k$, where $N_i = \cc w_i$.  Furthermore, since $z$ acts by a scalar, $U(\g)w_i \subseteq V$ is a simple Whittaker module, and thus ${\rm Wh}_\psi (U(\g)w_i) = \cc w_i$.  The result then follows from Proposition \ref{prop:decompInN(psi)}.
\end{proof}

In the following corollary, let $\ell (V)$ denote the composition length of $V$.  

\begin{cor}\label{cor:wVectsCompLength}
Let $V \in {\rm Ob} (\W_f(\psi))$ ($\psi \neq 0$), and assume $z \in \g$ acts by a scalar.  Then $\dim {\rm Wh}_\psi (V) \le \ell (V)$.  
\end{cor}

\begin{proof}
Let $\{ w_1, \ldots, w_k \}$ be a linearly independent subset of ${\rm Wh}_\psi (V)$, and let $W$ denote the $U(\g)$-submodule $\sum_{i=1}^k U(\g)w_i$ of $V$.  By Corollary \ref{cor:whittVectsDecomp}, we know that $W = \bigoplus_{i=1}^k U(\g)w_i$.  Since each summand $U(\g)w_i$ is simple, $\ell (W) = k$.  As $W \subseteq V$ is a submodule, we know $\ell (W) \le \ell (V)$, and the result follows.
\end{proof}

Note that Corollary \ref{cor:wVectsCompLength} and Lemma \ref{lem:compSeriesN} give us lower and upper bounds for the composition length of $V$.

\begin{prop}\label{prop:recognizeIndec-CR}
Fix $0 \neq \psi \in (\n^+/[n^+,\n^+])^*)$.  Suppose $V \in {\rm Ob} (\mathcal W_f(\psi))$ and that $z$ acts by some scalar.  Then 
\begin{enumerate}
\item $\dim {\rm Wh}_\psi (V) = 1$  if and only if every submodule of V is indecomposable, and 
\item $\dim {\rm Wh}_\psi (V) = \ell (V)$ (the composition length of $V$)  if and only if V is completely reducible.
\end{enumerate}
\end{prop}
\begin{proof}
For (1), assume $\dim {\rm Wh}_\psi (V) = 1$, and let $S \subseteq V$ be a submodule.  If $S = S_1 \oplus S_2$ for submodules $S_1$ and $S_2$, then by Theorem \ref{thm:Wdecomp}, $S_1$ and $S_2$ contain Whittaker vectors.  Such vectors must necessarily be linearly independent, so it follows from the assumption $\dim {\rm Wh}_\psi (V) = 1$ that $S$ cannot have such a decomposition.  Conversely, suppose that every submodule of $V$ is indecomposable, and let $W = U(\g) {\rm Wh}_\psi (V)$.  By Corollary \ref{cor:whittVectsDecomp}, $W$ is a direct sum of $\dim {\rm Wh}_\psi (V)$ simple Whittaker modules, so it must be that $\dim {\rm Wh}_\psi (V) =1$.  Statement (2) follows from Corollary \ref{cor:whittVectsDecomp} along with the fact that $\dim_\cc {\rm Wh}_\psi (L) = 1$ whenever $L$ is a simple Whittaker module of type $\psi \neq 0$.
\end{proof}

\subsection{Induced Modules}
The previous results in this section suggest that finite-dimensional $\n^+$-submodules may provide a tool for better understanding the $Z$-semisimple modules in $\W_f(\psi, \xi)$.  With this in mind, we construct a new set of induced modules.

Fix $0 \neq \psi \in (\n^+/[n^+,\n^+])^*)$ and $\xi \in \cc$. Suppose that $N$ is a finite-dimensional $(\cc z \oplus \n^+)$-module so that $z$ acts by $\xi$ and $x-\psi(x)$ acts locally nilpotently for all $x \in \n^+$.  Define the module $V_N$ by 
\begin{align} \label{eqn:defV_N}
V_N = U(\g) \otimes_{U(\cc z \oplus \n^+)} N, 
\end{align}
and note that $V_N \in {\rm Ob}(\mathcal W(\psi, \xi))$.   Where there is no confusion, regard $N$ as a subspace of $V_N$ by identifying $N$ with $1 \otimes N$. Following Lemma \ref{lem:compSeriesN}, $N$ has a basis $v_1, v_2, \ldots, v_n$ so that
\begin{equation} \label{eqn:finitePsiFlag}
(x-\psi(x))v_1=0, \quad (x-\psi(x))v_j  \in {\rm span}_\cc \{ v_k \mid k < j \}
\end{equation}
for $j>1$ and $x \in \n^+$. Additionally, $V_N$ is a free left $U(\n^- \oplus \cc d_0)$-module, since $U(\g) \cong U(\n^- \oplus \cc d_0) \otimes_\cc U(\cc z \oplus \n^+)$ and $U(\cc z \oplus \n^+) N = N$.

The following result implies that every $Z$-semisimple module in $\W_f(\psi, \xi)$ is the homomorphic image of some $V_N$.


\begin{prop}\label{lem:universalPropN}
Let $V \in {\rm Ob}(\W_f(\psi, \xi))$ where $z$ acts by $\xi$, and let $N \subseteq V$ be a finite-dimensional $\n^+$-module such that $V = U(\g)N$.  Then there is a unique surjective $\g$-module homomorphism $\pi : V_N \to V$ with the property that $\pi (1 \otimes w) = w$ for all $w \in N$.
\end{prop}
\begin{proof}
Let $v_1, \ldots, v_n \in N$ be a basis for $N$ as in (\ref{eqn:finitePsiFlag}), and recall that $V_N$ is a free $U(\n^- \oplus \cc d_0)$-module with basis $1 \otimes v_1, \ldots, 1 \otimes v_n$.  Thus every element of $V_N$ can be uniquely written in the form $\sum_i b_i \otimes v_i$, where $b_i \in U(\n^- \oplus \cc d_0)$, and we define $\pi \left( \sum_i b_i \otimes v_i \right) = \sum_i b_i v_i$.

It is clear that $\pi$ is a surjective $\cc$-linear map with the property that $\pi ( b \otimes v) = bv$ whenever $b \in U(\n^- \oplus \cc d_0)$ and  $v \in N$.  It suffices to show that $\pi \left( x \sum_i b_i \otimes v_i \right) = x \pi \left( \sum_i b_i \otimes v_i \right)$ for all $x \in \cc z \oplus \n^+$.  This follows from the fact that $U(\g) \cong U(\n^- \oplus \cc d_0) \otimes U(\cc z \oplus \n^+)$.
\end{proof}

\begin{lem}\label{lem:compLengthV_N}
Let $N$ be a finite-dimensional $\n^+ \oplus \cc z$-module, where $z$ acts on $N$ by $\xi \in \cc$.  Then $\ell (V_N) = \dim_\cc N$.
\end{lem}
\begin{proof}
We identify $N$ with the subspace $1 \otimes N \subseteq V_N$.  Let $n = \dim_\cc N$, and define $v_1, \ldots, v_n$ as in  (\ref{eqn:finitePsiFlag}).  
For $i \in \{ 1, \ldots, n \}$, let $V_i = U(\g) {\rm span}_\cc \{ v_1, \ldots, v_i \}$, and note that $V_i$ is a $U(\g)$-submodule of $V_N$.  The quotient $V_i / V_{i-1}$ is generated by the $\psi$-eigenvector $v_i + V_{i-1}$.  Since $V_i$ is a free $U(\n^- \oplus \cc d_0)$-module with basis $v_1, \ldots, v_i$, the quotient $V_i/V_{i-1}$ is necessarily nonzero.  Moreover, the quotients are simple (Whittaker) modules since $z$ acts by $\xi$ on $N$ and thus on $V_N$.
\end{proof}

Using the basis $\{ v_1, \ldots, v_n \}$ for $N$ in (\ref{eqn:finitePsiFlag}), suppose the action of $\n^+$ on $N$ is given by the matrix 
\begin{equation}\label{eqn:matrixOfFunctionals}
\left( \begin{array}{ccccc} \psi & \alpha(1,2) & \cdots & \alpha(1,n-1) & \alpha(1,n) \\ 0 & \psi & \cdots & \alpha(2,n-1) & \alpha(2,n)  \\ 0& 0 & \ddots & \vdots & \vdots \\ 0&0&0& \psi & \alpha(n-1,n) \\ 0&0&0&0& \psi \end{array} \right)
\end{equation}
where $\alpha(i,j) \in (\n^+)^*$ for all $i, j$.  
It is evident that the structure of the module $V_N$ depends upon the various $\alpha(i,j) \in \left( \n^+ \right)^*$. The following subsections demonstrate how this matrix affects certain properties of $V_N$.

Before proceeding, we establish some additional notation.  Let $V_i$ be the submodule of $V_N$ generated by $\{v_1, \ldots, v_i\}$.  If $\gamma \in \left( \n^+ \right)^*$, we define 
$$i \gamma \in \left( \n^+ \right)^*\quad \text{by} \quad (i \gamma) (d_k) = k \gamma_k \quad \text{for all $k \in \z$}.$$

\subsubsection{Structure of $V_N$ when $\psi_2 \neq 0$}\label{subsec:V_N-psi_2-not-0}

We see in Theorem \ref{thm:indecompSupDiag} that under certain conditions on $\alpha(i,j) \in \left( \n^+ \right)^*$, the induced module $V_N$ is indecomposable and non-simple. First we present two technical lemmas; the first follows from a direct computation.

\begin{lem}  \label{lem:d_0}
For $i, k>0$ and $w$ a Whittaker vector, 
$$
[d_i, d_0^k] w = \psi_i \left( (d_0+i)^k -d_0^k\right) w.
$$
\end{lem}

\begin{lem}\label{lem:computationInSimple}
Let $\psi \in (\n^+/[\n^+,\n^+])^*)$ such that $\psi_2 \neq 0$, and let $\xi \in \cc$.  Let $L(\psi, \xi)$ be the simple Whittaker module of Proposition \ref{cor:SimpleL_PsiXi} with cyclic Whittaker vector $w \in Wh_{\psi}(L(\psi, \xi))$.  If $v \in L(\psi, \xi)$ so that $(d_i - \psi_i)v \in \cc w$ for all $i>0$, then $v \in  \cc w \oplus  \cc d_0w$.  
\end{lem}
\begin{proof}
Suppose $v\in L(\psi, \xi)$ such that $(d_i - \psi_i)v \in \cc w$ for $i>0$.  Using the definition of $L(\psi, \xi)$ and the notation of Section \ref{subsec:simpObjW}, we can write
$$v = \sum_{\gamma \in \p, j \geq 0} a_{\gamma, j} d_{-\gamma} d_0^j w,$$
where $a_{\gamma, j} \in \cc$ with only finitely many $a_{\gamma,j} \neq 0$.
Note that
$$(d_i - \psi_i) v = \sum_{\gamma \in \p, j \geq	0} a_{\gamma, j} [d_i, d_{-\gamma}d_0^j] w.$$

Define $maxdeg(v) = max \{|\gamma| \mid a_{\gamma, j} \neq 0 \ \mbox{for some $j$}\}$. We first argue that $maxdeg(v)=0$.  Suppose not; that is, $maxdeg(v)=N>0$.  Let $\mathcal C = \{ (\lambda,i) \in \p \times \z_{\geq 0} \mid \mbox{$|\lambda| = N$ and $a_{\lambda, i} \neq 0$} \}$, and choose $(\lambda, i) \in \mathcal C$ with $i$ maximal.  We will show that $i=0$ and use this to derive a contradiction.  

Consider
\begin{eqnarray}\label{eqn:adD1}
[d_2, d_{-\lambda}d_0^i] w &=& [d_2, d_{-\lambda}] d_0^{i}w + d_{-\lambda} [d_2, d_0^i]w.
\end{eqnarray}
Note that $maxdeg ([d_2, d_{-\lambda}] d_0^iw) \leq N-1$. If $i \neq 0$, then (by Lemma \ref{lem:d_0}) the coefficient of $d_{-\lambda} d_0^{i-1} w$ in $[d_2, d_{-\lambda}d_0^i] w$ is nonzero.  We claim that the coefficient of $d_{-\lambda} d_0^{i-1} w$ in $(d_2 - \psi_{2}) v = \sum_{\gamma \in \p, j \geq 0} a_{\gamma} [d_2, d_{-\gamma}d_0^j] w$ is nonzero, as well.  For any $\gamma \in \p$ and $j \geq 0$, we have 
\begin{eqnarray}\label{eqn:adD1again}
[d_2, d_{-\gamma}d_0^j] w = [d_2, d_{-\gamma}] d_0^jw + d_{-\gamma} [d_2, d_0^j]w.
\end{eqnarray}
If $a_{\gamma,j} \neq 0$ then it is clear from our choice of $\lambda$ that $maxdeg ([d_2, d_{-\gamma}] d_0^jw) \le N-1$.  
Our choice of $(\lambda, i)$ guarantees that the term $d_{-\lambda} d_0^{i-1}w$ does not appear as a summand of $d_{-\gamma} [d_2, d_0^j]w$ for any $(\gamma, j)$ where $\gamma \neq \lambda$ or $j \neq i$. 
 However, $d_{-\lambda} d_0^{i-1} w \not\in \cc w$ since $|\lambda| =N > 0$.  This contradicts the choice of $v$, so it must be that $i=0$ if in fact $N>0$. 

Now $\lambda \neq \emptyset$ if $N>0$, so we may write $\lambda = (0 < \lambda_1 \le \cdots \le \lambda_s)$.  Consider the action of $d_{\lambda_1+ 2}$ on $v$.  In the proof of  \cite[Proposition 3.1]{OW2008}, it is shown that in this situation 
$$(d_{\lambda_1 +2} - \psi_{\lambda_1 + 2})v \neq 0$$
since $2 = \max \{ i \mid \psi_i \neq 0 \}$.  On the other hand, it is assumed that $(d_i - \psi_i)v \in \cc w$; this implies that $[\n^+, \n^+] v = 0$ so that $d_kv = 0$ for $k \ge 3$.  But then 
$$0=d_{\lambda_1+ 2}v= (d_{\lambda_1+ 2}-\psi_{\lambda_1+2} )v,$$
resulting in a contradiction.  Thus our original assumption that $maxdeg(v)=N>0$ must be incorrect. 

When $N=0$, Lemma \ref{lem:d_0} implies that the coefficient of $d_0^jw$ is zero for $j>1$, which completes the proof.
\end{proof}

The following theorem, along with Proposition \ref{prop:recognizeIndec-CR}, implies that $V_N$ is indecomposable in certain cases.  This theorem also illustrates an interesting difference between the present setting and the classical setting of complex semisimple Lie algebras.  Theorem 4.3 of \cite{Ko78} asserts that in the classical setting, the composition length of a module $V$ belonging to a category analogous to $\W_f(\psi)$ is equal to the dimension of ${\rm Wh}_\psi (V)$.

\begin{thm}\label{thm:indecompSupDiag}
Assume $\psi_2 \neq 0$.  Let $N$ be as in (\ref{eqn:matrixOfFunctionals}) and $V_N$ as in (\ref{eqn:defV_N}), and assume that $\alpha (i, i+1) \not\in \cc i \psi$ for $1 \le i \le n-1$.  Then ${\rm Wh}_\psi (V_N) = \cc v_1$.
\end{thm}
\begin{proof}
The proof is by induction on $\dim_\cc N$, with the base case being $\dim_\cc N = 2$.  Suppose that $N$ corresponds to the matrix $\left( \begin{array}{cc} \psi & \alpha \\ 0 & \psi \end{array} \right)$ with basis $\{ v_1, v_2 \}$, and write $V_1 = U(\g) v_1$. Then by Proposition \ref{cor:SimpleL_PsiXi}, $V_1 \cong L(\psi, \xi)$.

If $w \in {\rm Wh}_\psi (V_N)$, then $w + V_1 \in {\rm Wh}_\psi (V_N / V_1 )$ and so $w + V_1 = c v_2 + V_1$ for some $c \in \cc$ by Proposition 3.2 of \cite{OW2008} and the simplicity of $V_N/V_1$.  Write $cv_2 = w + x$ for some $x \in V_1$.  Then $(d_i - \psi_i)cv_2 = (d_i - \psi_i)(w+x)$, which implies
\begin{align}\label{eqn:indSupDiagComp1}
c \alpha_i v_1 = (d_i - \psi_i) x \qquad \text{for all $i>0$}.
\end{align}
By Lemma \ref{lem:computationInSimple}, $x = c'v_1 + c'' d_0 v_1$ for $c', c'' \in \cc$, so (\ref{eqn:indSupDiagComp1}) becomes 
$$c \alpha_i v_1 = c'' i \psi_i v_1 \qquad \text{for all $i>0$}.$$
If $c \neq 0$, then $\alpha = \frac{c''}{c} i \psi \in \cc i \psi$; this contradicts our assumption on $\alpha$.  Therefore it must be that $c = 0$, and it follows that $w \in V_1 \cap {\rm Wh}_\psi (V_N) = \cc v_1$ as desired.
  
Now assume that $n = \dim_\cc N> 2$, and let $\{ v_1, \ldots, v_n \}$ denote the standard basis for $N$.  Our goal is to show that ${\rm Wh}_\psi (V_N) = \cc v_1$.  By induction, we may suppose that the result holds for any module of the form $V_M$ where $\dim_\cc M < n$.  Let $V_{n-1}$ denote the submodule generated by the vectors $v_1, \ldots, v_{n-1}$.  
By induction, we have ${\rm Wh}_\psi (V_{n-1}) = \cc v_1$.

Suppose that $w \in {\rm Wh}_\psi (V_N)$ with $w \not\in \cc v_1$, and let $W = U(\g)w$.  Since $W$ is a simple Whittaker module and $w \not\in \cc v_1 = {\rm Wh}_\psi (V_{n-1})$, it follows that $W \cap V_{n-1} = \{ 0 \}$ and thus $V_N = V_{n-1} \oplus W$.  Write $v_n \in N$ as $v_n = v+w'$ with $v \in V_{n-1}$ and $w' \in W$.  We know that $(d_i - \psi_i)v_n = \sum_{j=1}^{n-1} \alpha(j,n)_i v_j \in M \subseteq V_{n-1}$.  Thus we have 
$$\sum_{j=1}^{n-1} \alpha(j,n)_i v_j = (d_i - \psi_i) (v+w') = (d_i - \psi_i)v + (d_i - \psi_i)w'.$$
Since $(d_i - \psi_i)w' \in W$ and $(d_i - \psi_i)v_n = \sum_{j=1}^{n-1} \alpha(j,n)_i v_j \in V_{n-1}$, the directness of $V_N = V_{n-1} \oplus W$ implies that $(d_i - \psi_i)w' = 0$.  Consequently, we have $v \in V_{n-1}$ with 
\begin{equation}\label{eqn:actionProject_v_n}
(d_i - \psi_i)v = \sum_{j=1}^{n-1} \alpha(j,n)_i v_j.
\end{equation}
Let $- : V_N \to V_N / V_{n-2}$ denote the natural homomorphism.  If we apply $-$ to (\ref{eqn:actionProject_v_n}), we obtain 
$$(d_i - \psi_i) \overline{v} = \sum_{j=1}^{n-1} \alpha(j,n)_i \overline{v_j} = \alpha(n-1,n)_i \overline{v_{n-1}}.$$
However, $V_N / V_{n-2} \cong V_{M'}$, where $M'$ is the 2-dimensional $\n^+$-module corresponding to the matrix $\left( \begin{array}{cc} \psi & \alpha(n-1,n) \\ 0 & \psi \end{array} \right)$.  The vector $\overline{v_{n-1}}$ is a Whittaker vector in the module $V_N / V_{n-1}$, so the submodule $U(\g) \overline{v_{n-1}}$ is simple.  Since (by assumption) $\alpha(n-1,n) \not\in \cc i \psi$, it follows that $\overline{v} \not\in U(\g) \overline{v_{n-1}}$.  But this is a contradiction since $v \in V_{n-1} = U(\g)v_{n-1} + \cdots + U(\g)v_1$.  It follows that there does not exist $w \in {\rm Wh}_\psi (V_N)$ with $w \not\in \cc v_1$.
\end{proof}

\begin{cor}
Assume $\psi_2 \neq 0$.  Let $N$ be as in (\ref{eqn:matrixOfFunctionals}) and $V_N$ as in (\ref{eqn:defV_N}).  Then $V_N$ is uniserial if and only if $\alpha (i, i+1) \not\in \cc i \psi$ for $1 \le i \le n-1$.
\end{cor}
\begin{proof}
We first suppose that $\alpha (i, i+1) \not\in \cc i \psi$ for $1 \le i \le n-1$.  The proof that $V_N$ is uniserial is by induction on $\dim_\cc N$, with the base case being $\dim_\cc N = 2$.  We have seen that if $N$ corresponds to the matrix $\left( \begin{array}{cc} \psi & \alpha \\ 0 & \psi \end{array} \right)$, then $\dim {\rm Wh}_\psi (V_N)=1$ if and only if $\alpha \not\in \cc i \psi$.  Since every submodule of $V_N$ must contain a Whittaker vector, there is only one submodule of composition length 1. 

Assume that $\dim_\cc N > 2$, and keep the notation $\alpha(1,2), \ldots, \alpha(n-1, n)$ from above.  Let $V_1 = U(\g) v_1$.  The submodule $V_1$ is simple and by Theorem \ref {thm:indecompSupDiag} contains ${\rm Wh}_\psi (V_N)$, every submodule of $V_N$ contains $V_1$.  Consequently the nontrivial submodules of $V_N$ correspond to the submodules of the quotient module $V_N / V_1$.  But $V_N / V_1$ is isomorphic to the module $V_M$, where $M$ corresponds to the matrix 
\begin{equation}\label{eqn:quotientMatrixFunctionals}
\left( \begin{array}{ccccc} \psi & \alpha(2,3) & \cdots & \alpha(2,n-1) & \alpha(2,n) \\ 0 & \psi & \cdots & \alpha(3,n-1) & \alpha(2,n)  \\ 0& 0 & \ddots & \vdots & \vdots \\ 0&0&0& \psi & \alpha(n-1,n) \\ 0&0&0&0& \psi \end{array} \right)
\end{equation}
By induction, $V_N/V_1$ is uniserial, and thus the same is true of $V_N$.

We now show the converse, with the proof again being by induction on $\dim_\cc N$.  The base case is given by $\dim_\cc (N) = 2$.  We have seen that if $N$ corresponds to the matrix $\left( \begin{array}{cc} \psi & \alpha \\ 0 & \psi \end{array} \right)$, then $\dim {\rm Wh}_\psi (V_N)=1$ if and only if $\alpha \not\in \cc i \psi$.  If $V_N$ is uniserial, it must be that $\dim_\cc {\rm Wh}_\psi (V_N) < 2$, and thus $\alpha \not\in \cc i \psi$.  

Now suppose that $\dim_\cc N > 2$, and let $V_1 = U(\g)v_1$ be the unique simple (Whittaker) submodule of $V_N$ and $V_2 = U(\g)v_1 + U(\g)v_2$.  The submodule $V_2$ is uniserial and corresponds to the matrix $\left( \begin{array}{cc} \psi & \alpha(1,2) \\ 0 & \psi \end{array} \right)$, so by induction $\alpha(1,2) \not\in \cc i \psi$.  Similarly, the quotient $V_N / V_1$ is uniserial and corresponds to the matrix (\ref{eqn:quotientMatrixFunctionals}).  Thus by induction $\alpha(2,3), \alpha(3,4), \ldots, \alpha(n-1,n) \not\in \cc i \psi$.  
\end{proof}

\begin{cor}\label{cor:howManyWhittakersV_N}
Assume $\psi_2 \neq 0$.  Let $N$ and $\{ \alpha(i,j) \mid 1 \le i \le n-1, 2 \le j \le n \}$ be as in (\ref{eqn:matrixOfFunctionals}) and assume $\dim_\cc N \ge 2$.  Then $\dim_\cc {\rm Wh}_\psi (V_N) \le | \mathcal W | +1$, where $\mathcal W = \{ \alpha(k, k+1) \mid \alpha(k,k+1) \in \cc i \psi \}$.
\end{cor}
\begin{proof}
The proof is by induction on $\dim_\cc N$.  The result is clear if $\dim_\cc N = 2$.  Let $\{ v_1, \ldots, v_n \}$ be the basis for $N$ corresponding to the $\alpha(i,j)$.  Then $v_1$ is a Whittaker vector, so we may extend $\{ v_1 \}$ to a basis $\mathcal B$ for ${\rm Wh}_\psi (V_N)$.  Write $\mathcal B = \{ w_1, \ldots, w_m \}$, where $w_1 = v_1$.  

Note that if $\alpha(j,j+1) \not\in \cc i \psi$ for all $j$, then the result follows from Theorem \ref{thm:indecompSupDiag}.  Thus we assume that some $\alpha(j,j+1)$ belongs to $\cc i \psi$.  Let $k$ be minimal such that $\alpha(k,k+1) \in \cc i \psi$, and let $V_k = U(\g) v_1 + \cdots + U(\g) v_k$.  By Theorem \ref{thm:indecompSupDiag}, $\dim_\cc {\rm Wh}_\psi (V_k) = 1$, and the quotient $\overline{V_N} = V_N / V_k$ corresponds to the matrix 
$$\left( \begin{array}{ccccc} \psi & \alpha(k+1,k+2) & \cdots & \alpha(k+1,n-1) & \alpha(k+1,n) \\ 0 & \psi & \cdots & \alpha(k+2,n-1) & \alpha(k+2,n)  \\ 0& 0 & \ddots & \vdots & \vdots \\ 0&0&0& \psi & \alpha(n-1,n) \\ 0&0&0&0& \psi \end{array} \right)$$
By induction, 
$$\dim_\cc {\rm Wh}_\psi (V_N / V_k) \le | \mathcal{W} \setminus \{ \alpha(k,k+1) \} | +1.$$
But since $\dim_\cc {\rm Wh}_\psi (V_k) = 1$, the set $\{ \overline{w_2}, \ldots, \overline{w_m} \} \subseteq {\rm Wh}_\psi (V_N / V_k)$ is linearly independent, and thus we have 
$$m-1 \le \dim_\cc {\rm Wh}_\psi (V_N / V_k) \le | \mathcal{W} \setminus \{ \alpha(k,k+1) \} | +1 = | \mathcal W | - 1 +1 = | \mathcal W |.$$
The result now follows since $m = \dim_\cc {\rm Wh}_\psi (V_N)$ and $m \le | \mathcal W | + 1$.
\end{proof}

\bigskip
\subsubsection{Structure of $V_N$ when $\psi_2 = 0$ and $\psi_1 \neq 0$}\label{subsec:V_N-psi_2=0}
This section presents partial descriptions of $V_N$ when $\psi_2 =0$ and $\psi_1 \neq 0$.  These results suggest that the structure of $W_f(\psi)$ is significantly different than when $\psi_2 \neq 0$.  We begin by defining some notation that will be useful in this setting.  If $\psi : \n^+ \to \cc$ is an algebra homomorphism, we define an algebra homomorphism 
$$\tilde \psi : \n^+ \to \cc$$
by
$$\text{$\tilde \psi (d_1) = \psi_1, \quad \tilde \psi (d_2) = -3 \psi_1^2$, \quad and \; $\tilde \psi (d_k) = 0$ for $k \ge 3$.}$$
Note that $(d_i - \psi_i)v_1 = \tilde \psi_i w$ and $(d_i - \psi_i)v_2 = i \psi_i w$ for $i>0$.  Note that $\{ \tilde \psi, i \psi \}$ is a linearly independent subset of ${\rm Hom}_{Alg}(\n^+, \cc)$ since $\tilde \psi_2 = -3 \psi_1^2 \neq 0$ and $2 \psi_2 = 0$.  Since ${\rm Hom}_{Alg}(\n^+, \cc)$ is 2-dimensional, it follows that $\{ \tilde \psi, i \psi \}$ spans ${\rm Hom}_{Alg}(\n^+, \cc)$.

\begin{lem}\label{lem:preWhittakerPsi2=0}
Let $L$ be a simple Whittaker module of type $\psi \neq 0$ with $w \in {\rm Wh}_\psi (L)$, and assume $\psi_2 = 0$.  If $v \in L$ with $(d_i - \psi_i)v \in \cc w$ for all $i>0$, then 
$$v \in \cc (d_0^2w - \psi_1 d_{-1}w) \oplus \cc d_0w \oplus \cc w.$$
\end{lem}
\begin{proof}
Let $v_1 = d_0^2w - \psi_1 d_{-1}w$ and $v_2 = d_0w$.  
If $(d_i - \psi_i)v \in \cc w$ for $i>0$, we may write $(d_i - \psi_i)v = \mu_i w$, where $\mu \in {\rm Hom}_{Alg}(\n^+, \cc) = {\rm span}_\cc \{ \tilde \psi, i \psi \}$.  Consequently we have $\mu = m \tilde \psi + n i \psi$, for $m,n \in \cc$, and it follows that the vector $w' = v - mv_1 - nv_2$ is a Whittaker vector.  This forces $w'  = pw$ for some $p \in \cc$, and so $v = mv_1 + nv_2 + pw$, as desired.
\end{proof}

\begin{prop}\label{lem:psi_2=0_length2}
Let $\psi : \n^+ \to \cc$ be a nonzero algebra homomorphism with $\psi_2 = 0$.  Let $N$ be the 2-dimensional $\n^+$-module corresponding to the matrix $\left( \begin{array}{cc} \psi & \alpha  \\ 0 & \psi \end{array} \right)$.  Then $V_N$ is completely reducible. 
\end{prop}
\begin{proof}
Let $\{ v_1, v_2 \}$ be the basis of $V_N$ corresponding to the matrix $\left( \begin{array}{cc} \psi & \alpha  \\ 0 & \psi \end{array} \right)$.  Since ${\rm span}_\cc \{ i \psi, \tilde \psi \} = {\rm Hom}_{Alg}(\n^+, \cc)$, we may write $\alpha = c_0 \tilde \psi + c_1 i \psi \in \cc \tilde \psi + \cc i \psi$.  Let $w_1 = v_1$ and 
$$w_2 = v_2 - c_0 \left( d_0^2v_1 - \psi_1 d_{-1}v_1 \right) - c_1 d_0 v_1.$$  
It is straightforward to show that $(d_i - \psi_i)w_2 = 0$ for all $i>0$, and thus $\dim_\cc {\rm Wh}_\psi (V_N) = 2$. The result then follows from Proposition \ref{prop:recognizeIndec-CR} and Lemma \ref{lem:compLengthV_N}.
\end{proof}

\begin{cor}\label{cor:V_N_ifPsi=0}
Let $\psi : \n^+ \to \cc$ be a nonzero algebra homomorphism with $\psi_2 = 0$, and let $N$ be as in (\ref{eqn:matrixOfFunctionals}).  Then $\dim {\rm Wh}_\psi (V_N) > 1$ and $V_N$ is not uniserial.
\end{cor}
\begin{proof}
Note that the submodule $V_2 = U(\g) v_1 + U(\g) v_2$ is completely reducible by Lemma \ref{lem:psi_2=0_length2}.
\end{proof}

\end{document}